\documentclass[11pt]{amsart}
\usepackage{amsmath}
\usepackage{amssymb}
\usepackage{amsfonts}
\usepackage{amsthm}
\usepackage{array}

\theoremstyle{plain}
\numberwithin{equation}{section}
\newtheorem{Theorem}{Theorem}
\newtheorem{Lemma}[Theorem]{Lemma}
\newtheorem{Proposition}[Theorem]{Proposition}

\theoremstyle{remark}

\date{}

\title[Asymptotics of spectral gaps]
{Asymptotics of spectral gaps of 1D Dirac operator with two
exponential terms potential}

\author{Berkay Anahtarci}

\address{Sabanci University, Orhanli,
34956 Tuzla, Istanbul, Turkey}
 \email{berkaya@sabanciuniv.edu}

\author{Plamen Djakov}

\address{Sabanci University, Orhanli,
34956 Tuzla, Istanbul, Turkey}
 \email{djakov@sabanciuniv.edu}

\begin{document}

\begin{abstract}
The one-dimensional Dirac operator
\begin{equation*}
L = i \begin{pmatrix} 1 & 0 \\ 0 & -1 \end{pmatrix} \frac{d}{dx}
+\begin{pmatrix} 0 & P(x) \\ Q(x) & 0 \end{pmatrix},
\quad  P,Q \in L^2 ([0,\pi]),
\end{equation*}
considered on $[0,\pi]$ with periodic and antiperiodic boundary
conditions, has discrete spectra. For large enough $|n|,\, n \in
\mathbb{Z}, $ there are two (counted with multiplicity) eigenvalues
$\lambda_n^-,\lambda_n^+ $ (periodic if $n$ is even, or antiperiodic
if $n$ is odd) such that $|\lambda_n^\pm - n |<1/2.$

We study the asymptotics of spectral gaps $\gamma_n =\lambda_n^+ -
\lambda_n^-$ in the case $$P(x)=a e^{-2ix} + A e^{2ix}, \quad Q(x)=b
e^{-2ix} + B e^{2ix},$$ where $a, A, b, B$ are nonzero complex
numbers. We show, for large enough $m,$ that $\gamma_{\pm 2m}=0 $ and
\begin{align*}
\gamma_{2m+1} = \pm 2 \frac{\sqrt{(Ab)^m (aB)^{m+1}}}{4^{2m} (m!)^2 }
\left[  1 + O \left( \frac{\log^2 m}{m^2}\right) \right],
\end{align*}
\begin{align*}
\gamma_{-(2m+1)} = \pm 2\frac{\sqrt{(Ab)^{m+1} (aB)^m}}{4^{2m} (m!)^2}
\left[  1 + O \left( \frac{\log^2 m}{m^2}\right) \right].
\end{align*}

\end{abstract}
\maketitle

\section{Introduction}

Consider one-dimensional Dirac operators of the form
\begin{equation}\label{i1}
Ly= i \begin{pmatrix} 1 & 0 \\ 0 & -1 \end{pmatrix}  \frac{dy}{dx} +
v(x)y, \quad  y= \begin{pmatrix} y_1  \\ y_2 \end{pmatrix}, \quad
 v=\begin{pmatrix} 0 & P \\ Q & 0 \end{pmatrix}
\end{equation}
with $\pi$-periodic complex valued functions  $P,Q \in L^2([0,
\pi]).$ The operator $L$ is symmetric if and only if
$Q(x)=\overline{P(x)};$ then  $L$ gives rise to a self-adjoint
operator in $L^2 (\mathbb{R}, \mathbb{C}^2)$  whose spectrum is
absolutely continuous and has a band-gap structure, i.e., $Sp(L)=
\mathbb{R} \setminus \bigcup_{n \in \mathbb{Z}}(\lambda_n^-,
\lambda_n^+)$.  The points $\lambda_n^-, \lambda_n^+$ are eigenvalues
of the same operator $L$ subject to periodic ($Per^+$) or
antiperiodic ($Per^-$) boundary conditions:
\begin{equation*}
Per^+: y(\pi)=y(0); \quad Per^-: y(\pi)=-y(0)
\end{equation*}
(see \cite{BESch, LS} for more details).

It is known that the potential smoothness determines the asymptotic
behavior of the sequence of {\em spectral gaps} $\gamma_n =
\lambda_n^+ - \lambda_n^-.$  Moreover, in the self-adjoint case the
asymptotic behavior of $(\gamma_n) $ determines the potential
smoothness as well. This phenomenon was first discovered and studied
for Hill-Schr\"odinger operators (see \cite{Hoch63, Hoch65, MO, MT,
Tr, KaMi01, DM3, DM15}). The situation is similar for self-adjoint
Dirac operators but the relationship between the smoothness of
potential functions $P, Q$ and the decay rate of spectral gaps
$\gamma_n $ has been studied later \cite{GKM, GK1, Mit1, Mit2, DM7,
DM6, DM15}.

In the non-self-adjoint case, for both Hill-Schr\"odinger and Dirac
operators, the  decay rate of $(|\gamma_n|)$ does not determine the
potential smoothness as Gasymov's example \cite{Gas} and its
modifications in the Dirac case show. However Tkachenko \cite{Tk92,
Tk94, Tk01} discovered that the potential smoothness could be
determined by the rate of decay of $(|\gamma_n|+ |\delta_n| ),$ where
$\delta_n $ is the difference between $\lambda_n^+$ and the closest
Dirichlet eigenvalue $\mu_n $ (see also \cite{ST96, DM5, DM7, DM15}).

Let us mention also the result of  Harrell \cite{Har}, Avron and
Simon \cite{AS} who has found the asymptotics of spectral gaps for
the Mathieu-Hill operator $L= -d^2/ dx^2 + 2 a \cos (2x)$, $a \in
\mathbb{R}$. They showed that
\begin{equation*}
\label{HAS} \gamma_n= \frac{8(|a|/4)^n}{[(n-1)!]^2} ( 1 + \rho_n)
\quad \text{with} \;\; \rho_n = o(1/n^2).
\end{equation*}
Recently \cite{AD} we have refined their result  by proving that
$\rho_n = - \frac{a^2}{4n^3} + O \left (\frac{1}{n^4}  \right ). $
See also \cite{DM10} for results in the case of two term
trigonometric polynomial potential.

Djakov and Mityagin \cite{DM8, DM9} studied the spectral gaps of
Dirac operators with potentials
$$
v(x) =\begin{pmatrix} 0 & P(x) \\ Q(x) & 0 \end{pmatrix}, \quad P(x)=
Q(x)= 2a \cos(2x), \;\; a \in \mathbb{R}\setminus \{0\},
$$
and showed that $\gamma_{-n} = \gamma_n \;\; \forall n,\;$ $\gamma_n
= 0 $ for even $n,$   and for $n=2m+1$ with $m>0$
\begin{equation}
\label{i3} \gamma_{2m+1}= 2|a|   \frac{a^{2m}}{4^{2m} (m!)^2}
 \left[ 1 + O \left(
\frac{\log m}{m} \right) \right], \quad m \to \infty.
\end{equation}
Let us note that here the operator $L$ is considered on the interval
$[0, \pi],$ whereas all  operators in \cite{DM8, DM9} are considered
on  $[0,1],$ and thus the coefficients in (\ref{i3}) are normalized
correspondingly.

In this paper, we study the asymptotics of spectral gaps $\gamma_n =
\lambda_n^+ - \lambda_n^-$ for (non-self-adjoint) Dirac operators
(\ref{i1}) with
\begin{equation*}
P(x)=a e^{-2ix} + A e^{2ix}, \quad  Q(x)=b e^{-2ix} + B e^{2ix},
\quad a, A, b, B \in \mathbb{C} \setminus \{0\}.
\end{equation*}
Our asymptotic formulas refine (\ref{i3}) in the case $a=A=b=B\in
\mathbb{R}$  (see Abstract). The main part of these asymptotics have
been given in \cite[(8.5) in Theorem~29]{DM31} but formula (8.5)
there is based on \cite[Proposition~28]{DM31} which is given without
a proof. We prove a refined version of that proposition in Section 4
(see Propositions~15 and 16). Essentially our approach is the same as
in \cite{DM8}, but we do a more precise asymptotic analysis and
overcome some additional difficulties that arise in the case of
non-self-adjoint operators.

\section{Preliminaries}

The Dirac operator (\ref{i1}), considered on $[0,\pi]$ with periodic
($Per^+$) or antiperiodic ($Per^-$) boundary conditions, gives a
rise to closed operators $L_{Per^\pm} (v)$ acting in $L^2
([0,\pi]).$ The following is well-known (e.g., \cite[Theorem
17]{DM15}).

\begin{Lemma}\label{loc}
The spectra of $L_{Per^\pm} (v)$ are discrete. There is $N_0=N_0
(v)$ such that the union $\cup_{|n|>N_0} D_n,  $ where $D_n =\{z: \,
|z-n|< \frac{1}{2}\},$ contains all but finitely many of the
eigenvalues of $L_{Per^\pm}(v).$

Moreover each disc $D_n, \, |n|>N_0,$ contains exactly two (counted
with algebraic multiplicity) periodic (if $n$ is even) or
antiperiodic (if $n$ is odd) eigenvalues $\lambda_n^-, \lambda_n^+$
(where $Re \,\lambda_n^- < Re \,\lambda_n^+$  or $Re \,\lambda_n^- =
Re \,\lambda_n^+$ and $Im \,\lambda_n^- \leq Im \,\lambda_n^+).$
\end{Lemma}

{\em Remark.} In the sequel we assume that $N_0>1$ and consider only
integers $n\in \mathbb{Z}$ with $|n|>N_0$.

In view of Lemma \ref{loc},
\begin{equation}
|\lambda_n^{\pm}-n|<1/2, \quad \text{for } |n|> N_0.
\end{equation}

Technically,  our approach is based on the following lemma
 (see \cite[Section 2.4]{DM15}).
\begin{Lemma}
\label{lem2} Let $ v= \begin{pmatrix} 0 & P \\ Q & 0 \end{pmatrix}
$, and let  $p(m)$ and $q(m), \, m  \in 2\mathbb{Z}$ be respectively
the Fourier coefficients of $P$ and $Q$ about the system $\{e^{imx},
\, m \in 2\mathbb{Z}\}.$  Then, $\lambda =n + z $ with $|z|\leq 1/2$
is an eigenvalue of $L_{Per^\pm} (v)$ if and only if $z$ is an
eigenvalue of a matrix $\begin{pmatrix} S^{11} & S^{12}
\\ S^{21} & S^{22}
\end{pmatrix}$ which entrees $S^{ij}=S^{ij} (n,z;v)$ are given by
\begin{equation}\label{S_ij}
S^{ij}(n,z)= \sum_{k=0}^{\infty}S_k^{ij}(n,z),
\end{equation}
where
\begin{equation}\label{S_0}
S_0^{11}= S_0^{22}=0, \quad S_0^{12}=p(-2n), \quad S_0^{21}=q(2n),
\end{equation}
and for $\nu =1, 2, \ldots $
\begin{equation}\label{odd}
S_{2\nu}^{11}= S_{2\nu}^{22}=0, \quad S_{2\nu -1}^{12}=S_{2\nu
-1}^{21}=0,
\end{equation}
\begin{align} S_{2\nu-1}^{11} &= \sum_{j_1,\ldots,j_{2\nu-1}\neq
n}\label{S_11} \frac{p(-n-j_1)q(j_1+j_2)\cdots p(-j_{2\nu-2}-j_{2\nu
-1})q(j_{2\nu -1}+n)}
{(n-j_1+z)(n-j_2+z)\cdots(n-j_{2\nu-2}+z)(n-j_{2\nu-1}+z)},\\
\label{S_22} S_{2\nu-1}^{22} &= \sum_{j_1,\ldots,j_{2\nu-1}\neq n}
\frac{q(n+j_1)p(-j_1-j_2)\cdots
q(j_{2\nu-2}+j_{2\nu-1})p(-j_{2\nu-1}-n)}
{(n-j_1+z)(n-j_2+z)\cdots(n-j_{2\nu-2}+z)(n-j_{2\nu-1}+z)}
\end{align}
\begin{align}
\label{S_12} S_{2\nu}^{12} &= \sum_{j_1,\ldots,j_{2\nu}\neq n}
\frac{p(-n-j_1)q(j_1+j_2)\cdots
q(j_{2\nu-1}+j_{2\nu})p(-j_{2\nu}-n)}{(n-j_1+z)(n-j_2+z)
\cdots(n-j_{2\nu-1}+z)(n-j_{2\nu}+z)},\\\label{S_21}
S_{2\nu}^{21} &= \sum_{j_1,\ldots,j_{2\nu}\neq n}
\frac{q(n+j_1)p(-j_1-j_2)\cdots
p(-j_{2\nu-1}-j_{2\nu})q(j_{2\nu}+n)}{(n-j_1+z)(n-j_2+z)
\cdots(n-j_{2\nu-1}+z)(n-j_{2\nu}+z)},
\end{align}
where in all sums  $j_k \in n +2\mathbb{Z}.$
\end{Lemma}

For each $\nu \in \mathbb{Z}_+$ the change of summation indices
$i_s= j_{2\nu+1-s}$, $s=1,\ldots, 2\nu +1$ shows that $S_{2 \nu
+1}^{11}(n,z)=S_{2 \nu +1}^{22}(n,z);$ therefore,
\begin{equation}\label{i}
S^{11}(n,z)=S^{22}(n,z).
\end{equation}

For convenience we set
\begin{equation}\label{ab}
\alpha_n(z):=S^{11}(n,z),\quad
\beta_n^+(z):=S^{21}(n,z),\quad \beta_n^-(z):=S^{12}(n,z).
\end{equation}
In these notations the characteristic equation associated with the
matrix $(S^{ij})$ becomes
\begin{equation}
\label{be}
(z-\alpha_n(z))^2=\beta_n^-(z)\beta_n^+(z).
\end{equation}
In view of Lemmas \ref{loc} and \ref{lem2},  for large enough $|n| $
equation (\ref{be}) has in the disc $|z|\leq 1/2$
 exactly the following two
roots (counted with multiplicity):
\begin{equation}
\label{zn}  z_n^- = \lambda_n^- -n, \quad z_n^+ = \lambda_n^+ -n.
\end{equation}

In the sequel we consider potentials of the form
$v(x)=\begin{pmatrix} 0 & P
\\ Q & 0 \end{pmatrix}$ with $P(x)=a e^{-2ix} + A e^{2ix}$ and
$Q(x)=b e^{-2ix} + B e^{2ix}.$ Then we have
\begin{equation}\label{pq}
p(-2)=a, \;\; p(2)=A; \quad q(-2)=b, \;\; q(2)=B,
\end{equation}
and
\begin{equation}\label{pq0}
\quad p(m)=q(m)=0 \;\;\text{for} \; m \neq \pm 2.
\end{equation}

Let us change in (\ref{S_21}) the indices $j_2, j_4, \ldots,
j_{2\nu}$ by $-j_2,  -j_4, \ldots, -j_{2\nu}.$ Then by \eqref{pq} and
\eqref{pq0} each nonzero term in the resulting sum comes from a
$2\nu$-tuple of indices $(j_1,\ldots, j_{2\nu})$ with $j_1, j_3,
\ldots, j_{2\nu-1} \neq n$ and $j_2, j_4, \ldots, j_{2\nu} \neq -n$
such that
\begin{equation}\label{w1}
(n+j_1)+(j_2-j_1)+ \cdots + (j_{2\nu}-j_{2\nu-1}) + (n-j_{2\nu})= 2n
\end{equation}
and
\begin{equation}\label{w2}
n+j_1, j_2-j_1, \ldots, j_{2\nu}-j_{2\nu-1}, n-j_{2\nu} \in \{ -2,2
\}.
\end{equation}
So by \eqref{S_ij}, \eqref{S_0} and \eqref{ab} we obtain that
\begin{equation}
\beta_n^+(z)=q(2n)+ \sum_{\nu=1}^\infty \mathcal{B}^+_{2\nu}(n,z),
\end{equation}
where
\begin{equation}\label{rho}
\mathcal{B}^+_{2\nu}= \sum_{(j_l)_{l=1}^{2\nu} \in I_{2\nu}}
\frac{q(n+j_1)p(j_2-j_1)\cdots p(j_{2\nu}-j_{2\nu-1})
q(n-j_{2\nu})}{(n-j_1+z)(n+j_2+z)\cdots(n-j_{2\nu-1}+z)(n+j_{2\nu}+z)},
\end{equation}
and
\begin{align}
\label{I}
I_{2\nu}= \{ (j_l)_{l=1}^{2\nu}:
j_1, j_3, \ldots, j_{2\nu-1} \neq n;
j_2, j_4, \ldots, j_{2\nu} \neq -n;\\ \nonumber
  n+j_1, j_2-j_1, \ldots, j_{2\nu}-j_{2\nu-1}, n-j_{2\nu} \in \{ -2,2 \} \}.
\end{align}

Recall that a walk $x$ on the integer grid $\mathbb{Z}$ from $a$ to
$b$ (where $ a, b \in \mathbb{Z}$) is a finite sequence of integers
$x=(x_t)_{t=1}^{\mu}$ with $x_1+x_2+ \ldots + x_\mu=b-a$. The
numbers
$$
j_k = a + \sum_{t=1}^k x_t, \quad  1\leq k < \mu
$$
are known as {\em vertices} of the walk $x.$

In view of  \eqref{w1} and \eqref{w2}, there is one-to-one
correspondence between the nonzero terms in \eqref{rho} and the
\emph{admissible} walks $x=(x_t)_{t=1}^{2\nu +1}$ on $\mathbb{Z}$
from $-n$ to $n$ with steps $x_t=\pm 2$ such that $j_1,  j_3,
\ldots, j_{2\nu -1} \neq n$ and $j_2, j_4, \ldots, j_{2\nu} \neq -n.
$ For every such walk $x=(x_t)_{t=1}^{2\nu +1}$  we set
\begin{equation}
\label{h(x,z)} h^+ (x,z)= \frac{q(x_1)p(x_2)q(x_3)\cdots p(x_{2\nu})
q(x_{2\nu
+1})}{(n-j_1+z)(n+j_2+z)\cdots(n-j_{2\nu-1}+z)(n+j_{2\nu}+z)}.
\end{equation}
Let $X_n(r), \; r=0, 1, 2, \ldots $ denote the set of all admissible
walks from $-n$ to $n,$ with $r$ negative steps  if $n>0$ or with
$r$ positive steps if $n<0.$ It is easy to see that every walk $x
\in X_n (r)$ has totally $|n| + 2r $  steps because $\sum x_t = 2n.
$ In these notations, we have
\begin{equation}\label{sigma_p}
\beta_n^+(z)= \sum_{r=0}^\infty \sigma^+_r(n,z) \quad \text{with} \quad
\sigma^+_r(n,z)=\sum_{x\in X_n(r)} h^+(x,z).
\end{equation}

Of course,  we may write similar formulas for $\beta_n^- (z) $ as
well.  A walk  $y=(y_t)_{t=1}^{2\nu +1}$ from $n$ to $-n$ is {\em
admissible} if its steps are $\pm 2$ and its vertices satisfy $j_1,
j_3,  \ldots, j_{2\nu -1} \neq  -n$ and $j_2, j_4, \ldots, j_{2\nu}
\neq n. $   We set
\begin{equation}
\label{h-} h^- (y,z)= \frac{p(y_1)q(y_2)p(y_3)\cdots q(y_{2\nu})
p(y_{2\nu +1})}{(n+j_1+z)(n-j_2+z)\cdots(n +
j_{2\nu-1}+z)(n-j_{2\nu}+z)},
\end{equation}
and let $Y_n(r), \; r=0, 1, 2, \ldots $ denote the set of all
admissible walks from $n$ to $-n$  having $r$ positive steps  if
$n>0$  or $r$ negative steps if $n<0.$ Then, changing in (\ref{S_12})
the indices $j_1, \ldots, j_{2\nu -1}$ by $-j_1,  \ldots, -j_{2\nu
-1},$ we see that
\begin{equation}\label{-sigma_p}
\beta_n^-(z)= \sum_{r=0}^\infty \sigma^-_r(n,z) \quad \text{with} \quad
\sigma^-_r(n,z)=\sum_{y\in Y_n(r)} h^-(y,z).
\end{equation}

Finally, we  consider $ \alpha_n (z).$  A walk $(w_t)_{t=1}^{2\nu} $
from $n $ to $n$ is {\em admissible} if its steps are $\pm 2 $ and
its vertices satisfy $j_1, \ldots, j_{2 \nu -1} \neq -n    $  and $
j_2, \ldots, j_{2\nu -2} \neq n. $ We set
\begin{equation}
\label{h} h (w,z)= \frac{p(w_1)q(w_2)\cdots p(w_{2\nu -1})
q(w_{2\nu})}{(n+j_1+z)(n-j_2+z)\cdots(n + j_{2\nu-2}+z)(n-j_{2\nu
-1}+z)},
\end{equation}
and let $W_n(\nu), \; \nu =1, 2, \ldots $ denote the set of all
admissible walks from $n$ to $n$  having $2\nu $ steps. In view of
\eqref{S_ij} and \eqref{ab}, changing in (\ref{S_11}) the indices
$j_1, \ldots, j_{2\nu -1}$ by $-j_1,  \ldots, -j_{2\nu -1},$ we
obtain that
\begin{equation}\label{tau}
\alpha_n(z)=\sum_{\nu=1}^\infty \tau_{\nu}(n,z) \quad \text{with}
\quad  \tau_{\nu}(n,z) = \sum_{w \in W_n (\nu)} h(w, z).
\end{equation}

Of course $\sigma_r^\pm $ and $\beta^\pm_n $ depend on the potential
functions but in the above notations this dependence is suppressed.
If we use instead the notations $\sigma_r^\pm (P,Q;n,z) $ and
$\beta_n^\pm (P,Q;z)$ then the following holds.
\begin{Lemma}
\label{lemPQ} In the above notations,
\begin{equation}
\label{PQsigma} \sigma^-_r (P,Q;n,z) =\sigma^+_r (Q,P;-n,-z), \quad
r\in \mathbb{Z}_+,
\end{equation}
and
\begin{equation}
\label{PQ} \beta_n^- (P,Q;z) = \beta_{-n}^+ (Q,P;-z).
\end{equation}
\end{Lemma}

\begin{proof}
Let us write also $ h^\pm_{P,Q} (x,z). $ One can easily see that $Y_n
(r) = X_{-n} (r) $ and if $y \in Y_n (r) $ then
$$
h^-_{P,Q} (y,z) = h^+_{Q,P} (y,-z).
$$
Now (\ref{PQsigma}) follows and so (\ref{PQ}) holds as well.

\end{proof}

\begin{Proposition}
\label{evenn} For $n\in 2\mathbb{Z}$ with large enough $|n|$ we have
\begin{equation}
\label{even1} \beta_n^- (z) \equiv 0, \quad \beta_n^+ (z) \equiv 0;
\end{equation}
\begin{equation}
\label{even2} z_n^* = \alpha (z_n^*),   \quad \text{where} \;\; z_n^*
= z_n^- = z_n^+;
\end{equation}
\begin{equation}
\label{even3}    \lambda_n^- = \lambda_n^+.
\end{equation}
\end{Proposition}

\begin{proof}
If $n$ is even then there are no admissible walks from $-n$ to $n.$
Indeed, since every admissible walk has odd number of steps equal to
$\pm 2,$ the sum of all steps is not divisible by  $4$ while $2n$ is
multiple to $4.$ Therefore, it follows that $\beta_n^+ (z) \equiv
0.$ The same argument shows that $\beta_n^- (z) \equiv 0,$ so
(\ref{even1}) is proved.

Now the equation (\ref{be}) takes the form $(z-\alpha_n (z))^2 =0, $
so it has a double root, say $z_n^*.$  Hence (\ref{even2}) and
(\ref{even3}) hold.
\end{proof}

\section{Estimates for $\alpha_n (z)$ and $z_n^\pm$}

In this section we give the asymptotics of $\alpha_n (z)$ and derive
asymptotic formulas for $z_n^\pm = \lambda_n^\pm -n$ using the basic
equation (\ref{be}).

The following lemma gives  preliminary asymptotic estimates of
$\beta_n^\pm (z) $ for odd $n \in \mathbb{Z};$ the precise
asymptotics will be given in the next section.
\begin{Lemma}
\label{lem4} If  $n=\pm (2m+1), \; m \in \mathbb{N}$  then
\begin{equation}
\label{3.1} \beta_n^\pm (z) = O \left  ( (8D^2)^m/m^m \right ), \quad
|z| \leq 1/2,
\end{equation}
where $D= \max \{|a|, |A|, |b|, |B|\}.$
\end{Lemma}

\begin{proof}
We prove (\ref{3.1}) for $ \beta_n^+ $ only. The same argument could
be used in the case of $ \beta_n^- $ as well, but by (\ref{PQ}) the
assertion for $ \beta_n^- $ follows if (\ref{3.1}) is known for $
\beta_n^+. $

Fix $r \in \mathbb{Z}_+,$  and  let $x \in X_n (r) $  be a walk from
$-n$  to $n$ having $r$  negative (positive) steps if $n$  is
positive (respectively negative). If $(j_\ell)_{\ell=1}^{2\nu},$  $
\, \nu = m + r, $ are the vertices of $x,$  then
\begin{equation}
\label{3.2} |n \pm j_{\ell}+z| \geq |n \pm j_{\ell}|-2^{-1} \geq
2^{-1}|n \pm j_{\ell}|, \quad  \ell = 1, \ldots, 2\nu.
\end{equation}
On the other hand we have
\begin{equation}
\label{3.3} |n - j_{\ell}| \cdot |n+j_{\ell+1}| \geq   |n|, \quad
\ell =1, \ldots, 2\nu-1.
\end{equation}
Indeed, both $|n- j_\ell | $ and $|n+ j_{\ell+1}| $ are even. If
$j_\ell$  and $j_{\ell+1} $ have the same sign, then at least one of
those numbers is greater than $|n|,$  so \eqref{3.3} follows. Since
$|j_{\ell+1} - j_\ell|=2,$ $ j_\ell$ and $ j_{\ell+1}$ could have
opposite signs if, and only if, $ |j_\ell|= |j_{\ell+1}|=1.$ But then
$$
|n - j_\ell| \cdot |n+j_{\ell+1}|= n^2 - 1 > |n|,
$$
so \eqref{3.3}  holds.  Now \eqref{3.2} and \eqref{3.3} imply, for
$n=\pm(2m+1)$ and $|z|\leq 1/2,$ that
$$
\frac{1}{|n-j_1+z||n+j_2+z|\cdots|n-j_{2\nu-1}+z||n+j_{2\nu}+z|} \leq
\frac{2^{2\nu}}{(2m)^\nu},
$$
so in view of \eqref{h(x,z)} we obtain
$$
|h^+ (x,z)| \leq  D^{2\nu +1} (2/m)^{\nu}, \quad \nu = m+r.
$$
Since the steps of every walk $x\in X_n (r) $ are equal  to $\pm 2,$
we have $card \,X_n (r) \leq 2^{2\nu}. $ Thus,
$$
|\sigma^+_r (n, z)|  \leq \sum_{x\in X_n (r)}  |h^+ (x,z)|  \leq  D
\left (8D^2/m \right )^{m+r},
$$
which implies (\ref{3.1}).

\end{proof}

\begin{Proposition}
\label{propz} For odd $n \in \mathbb{Z} $ with large enough $|n|$
\begin{equation}
\label{zn1} z_n^\pm = \alpha_n (n, z_n^\pm) + O(|n|^{-p}) \;\;
\forall p>0.
\end{equation}
\end{Proposition}

\begin{proof}
Let $n = \pm (2m+1). $  We know that $z_n^\pm $ are roots of
equation \eqref{be}. Therefore, from \eqref{3.1} it follows that
$$
|z_n^\pm - \alpha_n (n, z_n^\pm)|  = O \left ((8D^2/m)^m \right )
$$
which implies (\ref{zn1}).
\end{proof}

\begin{Lemma}
\label{lem6} For $n \in \mathbb{Z} $ with large enough $|n|$
\begin{equation}\label{alpha}
\alpha_n (z)=  \frac{Ab+aB}{2n}+ O\left( 1/n^2 \right), \quad |z|
\leq 1/2
\end{equation}
and
\begin{equation}\label{azn}
z_n^\pm =  \frac{Ab+aB}{2n}+ O\left( 1/n^2 \right).
\end{equation}

\end{Lemma}

\begin{proof}

We estimate $\alpha_n (z) $ by using (\ref{tau}). To evaluate
$\tau_1 (n,z) $ we consider the two-step walks from $n$ to $n.$
There are two such walks, respectively with steps $(2, -2)$ and
$(-2, 2),$ and the corresponding vertices are $j_1= n+2 $ and $j_1 =
n-2. $ Therefore, for  $|z|\leq 1/2 $ we have
\begin{equation}
\label{tau1} \tau_1(n,z)=\frac{Ab}{2n+2+z} + \frac{aB}{2n-2+z},
\end{equation}
which implies
\begin{equation}
\label{ineq4} \tau_1(n,z)= \frac{Ab+aB}{2n}+ O\left( 1/n^2 \right),
\quad |z| \leq 1/2.
\end{equation}

Next we consider $\tau_2 (n,z).$ The related set $W_n (2)$ of
four-step walks from $n$ to $n$ has two elements: $(2, 2, -2, -2)$
and $(-2, -2, 2, 2).$  The corresponding vertices are
\begin{equation*}
j_1=n+2, \quad j_2=n+4, \quad j_3=n+2
\end{equation*}
and
\begin{equation*}
j_1=n-2, \quad j_2=n-4, \quad j_3=n-2.
\end{equation*}
Therefore, in view of (\ref{h})
\begin{align}
\label{tau2}
\tau_2(n,z)&= \frac{a b A B}{[n+(n+2)+z][n-(n+4)+z][n+(n+2)+z]}\\
\nonumber &+ \frac{a b A B}{[n+(n-2)+z][n-(n-4)+z][n+(n-2)+z]},
\end{align}
so it follows that
\begin{equation}\label{ineq10}
\tau_2(n,z)=O\left( 1/n^2 \right), \quad |z|\leq 1/2.
\end{equation}

Further, if $w \in W_n (\nu), \; \nu =3, 4, \ldots   $ is a walk with
$2 \nu$ steps from $n$ to $n,$ then $h(w,z) $ is a fraction which
denominator $d(w,z) $ has the form
$$
d(w,z) =(2n\pm 2 +z) \prod_{k=1}^{\nu -1}
(n-j_{2k}+z)(n+j_{2k+1}+z).
$$
For $|z| \leq 1/2, $ we have  $|2n\pm 2 +z| \geq |n|/2$ and by
(\ref{3.2}) and (\ref{3.3}) the absolute value of every factor of the
product is greater than $ |n|/2, $ so
$$
|d(w,z)| \geq  (|n|/2)^\nu.
$$
Now the same argument as in the proof of Lemma \ref{lem4} leads to
\begin{equation}\label{ineq11}
|\tau_\nu (n,z)| \leq C^\nu/|n|^\nu, \quad \nu =3,4,\ldots,
\end{equation}
where $C$ is a constant depending only  on $a, b, A, B $. Therefore,
it follows that that
\begin{equation}\label{ineq12}
\sum_{\nu=3}^\infty |\tau_\nu (n,z)| \leq \sum_{\nu=3}^\infty
\frac{C^\nu}{|n|^{\nu}} = O\left( 1/|n|^3 \right) \quad \text{for}
\;\; |z|\leq 1/2.
\end{equation}
Now \eqref{ineq4}, \eqref{ineq10} and \eqref{ineq12} imply
(\ref{alpha}).  In view of (\ref{even2}) and (\ref{zn1}), (\ref{azn})
follows from (\ref{alpha}).
\end{proof}

Next we refine (\ref{azn}) by finding the next term in the asymptotic
expansion of $z_n^\pm $ about the powers of $1/|n|.$

\begin{Proposition}
For $n \in \mathbb{Z} $ with large enough $|n|$
\begin{equation}
\label{zz} z_n^\pm = \frac{Ab+aB}{2n}+ \frac{aB-Ab}{2n^2}+
 O\left(1/|n|^3  \right).
\end{equation}

\end{Proposition}

\begin{proof}
From (\ref{tau1}) and (\ref{azn}) it follows that
$$
\begin{aligned}
\tau_1 (n,z_n^\pm) &= \frac{Ab}{2n} (1 - 1/n + O(1/n^2))
+\frac{aB}{2n} (1
+ 1/n + O(1/n^2))\\
&=\frac{Ab+aB}{2n}+ \frac{aB-Ab}{2n^2} +
 O\left( 1/|n|^3 \right).
\end{aligned}
$$
On the other hand,  (\ref{tau2}) and (\ref{azn}) imply with $z=
z_n^\pm$
$$
\tau_2(n,z_n^\pm) = \frac{-abAB}{(2n+2+z)^2 (4-z)}
+\frac{abAB}{(2n-2+z)^2 (4+z)} =O\left( \frac{1}{|n|^3} \right).
$$
Therefore, in view of (\ref{ineq12}) we obtain (\ref{zz}).
\end{proof}

{\em Remark.} Of course, one can easily get more terms of the
asymptotic expansion of $z_n^\pm$ by using \eqref{zz} and refining
further the asymptotic analysis of $\alpha_n (z_n^\pm). $

To estimate $\gamma_n = \lambda_n^+ - \lambda_n^- = z_n^+ - z_n^-$
in the next section we need the following.
\begin{Lemma} \label{lem7}
If $n = \pm (2m+1)$ then
\begin{equation}
\label{7.1} \alpha^\prime_n (z) =  O \left( 1/m^2\right) \quad
\text{for} \quad |z| \leq 1/4
\end{equation}
and
\begin{equation}
\label{D1} \alpha_n(z_n^+)-\alpha_n(z_n^-)= \gamma_n O \left(
1/m^2\right).
\end{equation}
\end{Lemma}

\begin{proof}

By \eqref{tau} we have
$$
\alpha_n(z) =\tau_1(n,z) + \tilde{\alpha}_n(z) \quad \text{with}
\quad \tilde{\alpha}_n(z) = \sum_{\nu=2}^\infty \tau_\nu (n,z).
$$
In view of \eqref{ineq10} and \eqref{ineq12},
\begin{equation*}
\tilde{\alpha}_n(z) = O(1/m^2) \quad \text{for} \;\; |z|\leq 1/2.
\end{equation*}
Therefore, the Cauchy formula for derivatives implies that
\begin{equation*}
\tilde{\alpha}^\prime_n (z) =O(1/m^2) \quad \text{for} \;\; |z|\leq
1/4.
\end{equation*}
On the other hand, by \eqref{tau1} we have
$$
 \partial_z \tau_1(n,z)=
 -\frac{Ab}{(2n+2+z)^2} - \frac{aB}{(2n-2+z)^2}= O \left
 (\frac{1}{m^2} \right ) \quad \text{for} \;\; |z| \leq 1/2,
$$
so (\ref{7.1}) follows.

Further we have
\begin{align}\label{D2}
\alpha_n(z_n^+)-\alpha_n(z_n^-)=\int_{z_n^-}^{z_n^+}
\alpha^\prime_n(z) \, dz,
\end{align}
where the integral is taken over the segment $[z_n^-,z_n^+]$ from
$z_n^-$ to $z_n^+.$ Therefore, by (\ref{7.1}) we obtain
$$
|\alpha_n(z_n^+)-\alpha_n(z_n^-)| \leq  |z_n^+ - z_n^-|
\sup_{[z_n^-,z_n^+]} |\alpha^\prime_n (z)| = |z_n^+ - z_n^-| \,
O(1/m^2),
$$
hence \eqref{D1} holds.

\end{proof}

\section{Asymptotic formulas for $\beta_n^\pm(z)$ and $\gamma_n.$}

In this section only odd integers $n$ with large enough $|n|$ are
considered.

We use \eqref{sigma_p} to find  precise asymptotics of $\beta_n^+
(z)$. First we analyze $\sigma^+_0 (n,z). $  If $n = 2m+1 $ with
$m\in \mathbb{N}$ then there is only one admissible walk from $-n$ to
$n$ with no negative steps. We denote this walk  by $\xi, $  so we
have $X_n (0) = \{\xi\}$ and $\sigma_0 (n,z) = h^+ (\xi, z). $ Since
\begin{equation}
\label{xi} \xi(t) = 2, \quad 1 \leq t \leq 2m+1
\end{equation}
we obtain
\begin{equation}\label{xiz}
\sigma^+_0 (n,z) =\frac{A^m B^{m+1}}
{(n-j_1+z)(n+j_2+z)\cdots(n-j_{2m-1}+z) (n+j_{2m}+z)}
\end{equation}
with $j_{\nu}=-2m-1+2\nu$, $\; \nu=1,\ldots, 2m.$

If $n= -(2m+1),$ then again $X_n (0)$ has only one element, say
\begin{equation}
\label{barxi} \bar{\xi}=(\bar{\xi_t})_{t=1}^{2m+1}, \quad
\bar{\xi}(t)=-2 \;\; \forall t.
\end{equation}
Therefore $\sigma_0 (n,z) = h^+ (\bar{\xi}, z) $ and so it follows
that
\begin{equation}\label{barxiz}
\sigma^+_0 (n,z) =\frac{a^m b^{m+1}}
{(n-j_1+z)(n+j_2+z)\cdots(n-j_{2m-1}+z) (n+j_{2m}+z)}
\end{equation}
with $j_{\nu}=2m+1-2k$, $\; \nu=1,\cdots, 2m.$

\begin{Lemma} In the above notations,
\begin{equation}\label{Sigma0}
 \sigma^+_0(n,0)=
\begin{cases}
\frac{A^m B^{m+1}}{4^{2m} (m!)^2} & \text{for} \;\;  n=2m+1, \\
\frac{a^m b^{m+1}}{4^{2m} (m!)^2}, &  \text{for} \;\;  n=-2m-1.
\end{cases}
\end{equation}
\end{Lemma}

\begin{proof} In the case $n= 2m+1$ we have
$$
\prod_{k=1}^m [n-(-n+2(2k-1))] \prod_{k=1}^m [n+(-n+2(2k))]= 4^{2m}
(m!)^2,
$$
so (\ref{Sigma0}) holds. The proof is similar for $n= -2m-1.$
\end{proof}

It is well known (as a partial case of the
Euler-Maclaurin sum formula, see \cite[Sect. 3.6]{Br}) that
\begin{align}\label{HarmonicSum}
\sum_{k=1}^m \frac{1}{k}= \log m + g + \frac{1}{2m}
 + O\left( \frac{1}{m^2} \right), \quad m\in
\mathbb{N},
\end{align}
where $g= \lim_{m \to \infty}
\left( \sum_{k=1}^m \frac{1}{k} -\log{m} \right) $ is the Euler constant.

\begin{Lemma}
\label{lem9} For $n= \pm (2m+1),$
\begin{align}\label{b2}
\frac{\sigma^+_0 (n,z_n^\pm)}{\sigma^+_0(n,0)}= \left[1 -
\frac{(Ab+aB)\log m}{8m} - \frac{g(Ab+aB)}{8m} + O \left(
\frac{\log^2 m}{m^2} \right) \right].
\end{align}
\end{Lemma}

\begin{proof}
From (\ref{xiz}) and (\ref{barxiz}) it follows that
$$
\frac{\sigma^+_0 (n,z_n^\pm)}{\sigma^+_0(n,0)}=
\prod_{k=1}^{m}(1+c_k)^{-1} \prod_{k=1}^{m}(1+d_k)^{-1}
$$
with
\begin{equation}
\label{c_k}
 c_k= sgn \,(n) \, \frac{z_n^\pm}{4k},    \quad
d_k =  \frac{sgn \,(n)\, z_n^\pm}{4(m-k+1)}.
\end{equation}

One can easily see that $\prod_{k=1}^{m}
(1+c_k)^{-1}=\prod_{k=1}^{m} (1+d_k)^{-1}$ and
\begin{align*}
\log\left( \prod_{k=1}^{m} (1+c_k)^{-1} \right) &= -\sum_{k=1}^{m}
\log(1+c_k) = -\sum_{k=1}^{m} c_k + O \left( \sum_{k=1}^{m} |c_k|^2
\right).
\end{align*}
In view of \eqref{azn} and \eqref{c_k} we have
\begin{align*}
\sum_{k=1}^{m} c_k = \left( \sum_{k=1}^{m} \frac{1}{4k} \right)
\left[ \frac{Ab+aB}{4m} +  O\left( \frac{1}{m^2} \right) \right],
\end{align*}
and
\begin{equation*}
\sum_{k=1}^m |c_k|^2= \left( \sum_{k=1}^m \frac{1}{16k^2} \right)
 O\left( \frac{1}{m^2} \right) = O\left( \frac{1}{m^2} \right).
\end{equation*}
Therefore, by \eqref{HarmonicSum} we obtain
\begin{equation}
\label{sumc_k} \log\left( \prod_{k=1}^{m} \frac{1}{1+c_k} \right)=-
\frac{(Ab+aB) \log m}{16m}- \frac{g(Ab+aB)}{16m} + O\left(
\frac{\log m}{m^2} \right).
\end{equation}
Hence,
\begin{align*}
\prod_{k=1}^{m} \frac{1}{1+c_k} &= 1 -\frac{(Ab+aB) \log m}{16m}
-\frac{g(Ab+aB)}{16m} + O\left( \frac{\log^2 m}{m^2} \right),
\end{align*}
which implies \eqref{b2}.
\end{proof}

We need also the following modification of Lemma~\ref{lem9}.
\begin{Lemma}
\label{lem9a} For $n= \pm (2m+1),$  if $z= O(1/m)$  then
\begin{align}\label{mb2}
\sigma^+_0 (n,z) =\sigma^+_0(n,0) (1+O((\log m)/m)).
\end{align}
\end{Lemma}

\begin{proof}
We follow the proof of Lemma~\ref{lem9}, replacing $z_n^\pm$ by $z$
and using $z= O(1/m)$ instead of \eqref{azn}

\end{proof}

Next we estimate the ratio $\sigma^+_1(n,z)/\sigma^+_0(n,z).$
\begin{Lemma}
\label{lemR} If $n= \pm (2m+1),$ then
\begin{equation}
\label{r1} \sigma^+_1(n,z)=\sigma^+_0(n,z) \cdot \Phi (n,z),
\end{equation}
where
\begin{equation}
\label{r2}  \Phi (n,z)=\sum_{k=1}^{m} \varphi_k (n,z) +
\sum_{k=2}^{m} \psi_k (n,z)
\end{equation}
with
\begin{equation}
\label{r30}  \varphi_k (n,z)= \frac{bA}{(4(m+1-k) \pm z)(4k \pm  z)}
\end{equation}
and
\begin{equation}
\label{r3}  \psi_k (n,z)= \frac{aB}{(4(k-1) \pm z)(4(m+1-k) \pm z)},
\end{equation}
where we have $+$ or $-$ in front of $z$ if $n>0$ or $n<0$
respectively.
\end{Lemma}

\begin{proof}
From the definition of $X_n (1) $ and
 \eqref{sigma_p} it follows that
\begin{equation}\label{sigma1}
\sigma^+_1(n,z)=\sum_{x\in X_n(1)} h^+ (x,z)= \sum_{\nu=2}^{2m}
h(x_\nu,z),
\end{equation}
where $x_\nu$ denotes the walk with $(\nu+1)$-th step equal to -2 and
all others equal to 2 if $n= 2m+1$ or the walk with $(\nu+1)$-th step
equal to 2 and all others equal to -2 if $n=-( 2m+1).$ The vertices
of  $x_\nu$ are given by
\begin{equation*}
j_\alpha(x_\nu)= \begin{cases} i_\alpha, & 1\leq \alpha\leq \nu\\
i_{\nu-1}, &   \alpha = \nu+1\\
i_{\alpha-2}, &  \nu+2  \leq \alpha \leq |n|+2
\end{cases}
\quad \text{with} \;\;
i_k= \begin{cases}  -n + 2k,    &  n >0\\
-n - 2k,    &    n  < 0
\end{cases}
\end{equation*}
Therefore, by \eqref{h(x,z)}
\begin{equation} \label{b12}
h(x_{2k},z)= h(\xi,z)\frac{bA}{(n-i_{2k-1}+z)(n+i_{2k}+z)},
\end{equation}
and
\begin{equation} \label{b12b}
h(x_{2k-1},z)= h(\xi,z)\frac{aB}{(n+i_{2k-2}+z)(n-i_{2k-1}+z)}.
\end{equation}
Now (\ref{sigma1})--(\ref{b12b}) imply (\ref{r1}).
\end{proof}

\begin{Lemma}\label{lem11}
If $n=\pm (2m+1)$ and  $z=O(1/m),$ then
\begin{equation}
\label{r12} \Phi (n,z) = \frac{(Ab+aB) \log m}{8m}+
\frac{g(Ab+aB)}{8m} + O \left( \frac{\log m}{m^2} \right)
\end{equation}
and
\begin{equation}\label{r13}
\Phi^*(n,z) :=  \sum_{k=1}^{m} |\varphi_k (n,z)| + \sum_{k=2}^{m}
|\psi_k (n,z)|= O \left( \frac{\log m}{m} \right).
\end{equation}

\end{Lemma}

\begin{proof}
In view of \eqref{r2}--\eqref{r3},
\begin{equation*}
\Phi (n,z) = \frac{bA}{2m+2\pm z} \sum_{k=1}^m \frac{1}{4k\pm z} +
\frac{aB}{2m+2\pm z} \sum_{k=1}^{m-1} \frac{1}{4k\pm z}.
\end{equation*}
Since
\begin{equation*}
\sum_{k=1}^m \frac{1}{4k+z} = \sum_{k=1}^m \frac{1}{4k} + O \left(
\frac{1}{m} \right) \quad \text{if~ }  z = O(1/m),
\end{equation*}
\eqref{r12} follows immediately.

To obtain \eqref{r13} we use that $|4k \pm z| \geq 2k, $ and
therefore,
$$
\Phi_n^* (n,z) \leq 4 \Phi_n^* (n,0) \leq \frac{4|bA|}{2m+2}
\sum_{k=1}^m \frac{1}{4k} + \frac{4|aB|}{2m+2} \sum_{k=1}^{m-1}
\frac{1}{4k}= O \left( \frac{\log m}{m} \right).
$$
\end{proof}

\begin{Proposition}
\label{propbeta+} For $n=\pm(2m+1)$  we have
\begin{align}\label{betaz+}
\beta_n^+(z)= \sigma^+_0(n,0) \biggl [  1 + O \left( \frac{\log
m}{m}\right) \biggr]  \quad \text{if} \;\; z= O (1/m)
\end{align}
and
\begin{align}\label{beta+}
\beta_n^+(z_n^\pm) =\sigma^+_0(n,0)
\left[  1 + O \left( \frac{\log^2 m}{m^2}\right) \right],
\end{align}
with
\begin{equation}\label{Sigma0+}
 \sigma^+_0(n,0)=
\begin{cases}
\frac{A^m B^{m+1}}{4^{2m} (m!)^2} & \text{for} \;\;  n=2m+1, \\
\frac{a^m b^{m+1}}{4^{2m} (m!)^2}, &  \text{for} \;\;  n=-2m-1.
\end{cases}
\end{equation}
\end{Proposition}

\begin{proof}
From (\ref{mb2}), (\ref{r1}) and (\ref{r12}) it follows that
$$
\sigma^+_1(n,z)+\sigma^+_0(n,z) = \sigma^+_0(n,0)\left[ 1  + O \left(
\frac{\log m}{m}\right)\right] \quad \text{if} \;\; z= O (1/m).
$$
Also,  (\ref{b2}), (\ref{r1}) and (\ref{r12}) imply that
\begin{align*}
\sigma^+_1(n,z_n^\pm)+\sigma^+_0(n,z_n^\pm) =
\sigma^+_0(n,0)\left[  1  + O \left( \frac{\log^2 m}{m^2}\right)\right] .
\end{align*}
Since $\beta_n^+ (z) = \sum_{r=0}^\infty \sigma^+_r (n,z), $ to
complete the proof it is enough to show that
\begin{equation}
\label{S2} \sum_{r=2}^\infty \sigma^+_r(n,z)= \sigma^+_0(n,z) \,
O\left(\frac{\log^2 m}{m^2}\right) \quad \text{if} \;\; z= O (1/m).
\end{equation}
Next we prove (\ref{S2}). Recall that $\sigma^+_r (n,z) = \sum_{x\in
X_n(r)} h^+ (x,z).$ Now we set
$$
\sigma^*_r (n,z) := \sum_{x\in X_n(r)} |h^+ (x,z)|.
$$
We are going to show that there is an absolute constant $C>0$ such
that for $n=\pm (2m+1) $ with large enough $m$
\begin{equation}
\label{S3} \sigma^*_r (n,z) \leq \sigma^*_{r-1} (n,z) \cdot \frac{C
\log m}{m} \quad \text{if} \;\; z= O (1/m),  \;\; r= 1, 2, \ldots.
\end{equation}
Since $\sigma^+_0 (n,z)$ has one term only, we have $\sigma^*_0
(n,z)=|\sigma^+_0 (n,z)|.$

Let $r \in \mathbb{N}.$ To every walk $x \in X_n(r)$ we assign a pair
$(\tilde{x},k)$, where $k$ is such that $x(k+1) $ is the first
negative (if $n>0$) or positive (if $n<0$) step of $x$ and $\tilde{x}
\in X_n(r-1)$ is the walk that we obtain after dropping from $x$ the
steps $x(k) $ and $x(k+1). $ In other words, we consider the map
\begin{equation*}
\varphi: X_n(r)\longrightarrow X_n(r-1) \times I,   \quad
\varphi(x)=(\tilde{x},k), \;\; k \in I=\{2, \ldots, 2m\},
\end{equation*}
where $k=\begin{cases} \min\{t:x(t)=2,~x(t+1)=-2\}   &   \text{if}
\;\; n>0,\\
\min\{t:x(t)=-2,~x(t+1)=2\} &   \text{if} \;\; n< 0,
\end{cases}$
\begin{equation*}
\tilde{x}(t)=
\begin{cases}
x(t) & \text{if } 1\leq t\leq k-1, \\
x(t+2) & \text{if } k \leq t \leq 2m+2r-1.
\end{cases}
\end{equation*}
The map $\varphi$ is clearly injective and we have
\begin{equation*} \label{S5}
\frac{h(x,z)}{h(\tilde{x},z)}=
\begin{cases}
 \frac{bA}{(n-j+2\pm z)(n+j\pm z)} & \text{if $k$ is even}, \\
 \frac{aB}{(n+j-2\pm z)(n-j\pm z)} & \text{if $k$ is
odd},
\end{cases} \quad j= \begin{cases}   -n + 2k   &   \text{if} \;\; n>0,\\
-n-2k    &   \text{if} \;\; n<0,
\end{cases}
\end{equation*}
where in front of $z$ we have $+$ if $n>0$ or $-$ if $n<0.$

Since  $\varphi$ is injective, from (\ref{r3}), (\ref{r13}) it
follows that
\begin{equation*}\label{S6}
\sigma^*_r(n,z) \leq \sigma^*_{r-1}(n,z)\cdot \Phi^*(n,z).
\end{equation*}
Hence, by \eqref{r13} and (\ref{r12}), we obtain that (\ref{S3})
holds.

From (\ref{S3}) it follows (since $\sigma^*_0 (n,z)=|\sigma^+_0
(n,z)|$) that
$$
\sigma^*_r(n,z) \leq |\sigma^+_0 (n,z)| \cdot \left (\frac{C \log
m}{m}\right)^r.
$$
Hence, (\ref{S2}) holds, which completes the proof.

\end{proof}

The  asymptotics of $\beta_n^- $ could be found in a similar way. We
have the following.
\begin{Proposition}
\label{propbeta-} If $n=\pm(2m+1)$  then
\begin{align}\label{betaz-}
\beta_n^-(z)= \sigma^-_0(n,0) \biggl [  1 + O \left( \frac{\log
m}{m}\right) \biggr] \quad \text{if} \;\;  z= O (1/m)
\end{align}
and
\begin{align}\label{beta-}
\beta_n^-(z_n^\pm) =\sigma^-_0(n,0) \left[  1 + O \left( \frac{\log^2
m}{m^2}\right) \right],
\end{align}
with
\begin{equation}\label{Sigma0-}
 \sigma^-_0(n,0)=
\begin{cases}
\frac{a^{m+1} b^{m}}{4^{2m} (m!)^2} & \text{for} \; n=2m+1, \\
\frac{A^{m+1} B^{m}}{4^{2m} (m!)^2} & \text{for} \; n=-2m-1.
\end{cases}
\end{equation}
\end{Proposition}

\begin{proof}
One could give a proof by following step by step the proof of
Proposition~\ref{propbeta+} but analyzing the sums (\ref{-sigma_p})
instead of (\ref{sigma_p}).

However, Lemma~\ref{lemPQ} provides an alternative approach. In view
of (\ref{PQ}), formula (\ref{betaz-}) follows from (\ref{betaz+})
immediately.

\end{proof}

\begin{Theorem}
The Dirac operator \eqref{i1} considered with
$$P(x)=a e^{-2ix} + A e^{2ix}, \quad Q(x)=b e^{-2ix} + B e^{2ix},
 \quad a, A, b, B \in \mathbb{C} \setminus \{0\}, $$
has for $n\in \mathbb{Z}$ for large enough $|n|$ two periodic (if
$n$ is even) or anti-periodic (if $n$ is odd) eigenvalues
$\lambda_n^-$, $\lambda_n^+$ such that
\begin{equation}
\label{zz1} \lambda_n^\pm =n+ \frac{Ab+aB}{2n}+ \frac{aB-Ab}{2n^2}+
 O\left(1/|n|^3  \right).
\end{equation}
If $n$ is even, then $\gamma_n =\lambda_n^+ - \lambda_n^- =0. $ For
odd $n= \pm (2m+1), \; m \in \mathbb{N},$ we have
\begin{align}\label{gamma+}
\gamma_{2m+1} = \pm 2 \frac{\sqrt{(Ab)^m (aB)^{m+1}}}{4^{2m} (m!)^2
}\, \left[  1 + O \left( \frac{\log^2 m}{m^2}\right) \right].
\end{align}
and
\begin{align}\label{gamma-}
\gamma_{-(2m+1)} = \pm 2\frac{\sqrt{(Ab)^{m+1} (aB)^m}}{4^{2m} (m!)^2}
\left[  1 + O \left( \frac{\log^2 m}{m^2}\right) \right].
\end{align}
\end{Theorem}

\begin{proof}
For even $n$ with large enough $|n|$ we have $\lambda_n^+ =
\lambda_n^-$ by Proposition~\ref{evenn}, and \eqref{zz1} comes from
\eqref{zz}.

Let $n = \pm(2m+1), $ and let $$C= \max \{|a|^2, |b|^2, |A|^2,
|B|^2\}, \qquad  \mathbb{D}_m = \{z: \; |z| < C/m\}.$$ In view of
(\ref{azn}), for large enough $m$ we have
$$
|z_n^\pm |  \leq C/(2m),
$$
so $z_n^\pm \in \frac{1}{2}\mathbb{D}_m. $

On the other hand, from (\ref{betaz+}) and (\ref{betaz-}) it follows
that for large enough $m$
$$      \beta_n^\pm (z) = \sigma^\pm_0 (n,0) (1+ r^\pm_n (z)) \quad
\text{with}  \;\;  |r^\pm_n (z)| \leq 1/2 \quad \text{for} \;\;
z \in 2\mathbb{D}_m.
$$ We set
\begin{equation*}
\sqrt{\beta_n^- (z)\beta_n^+ (z)} := \sqrt{\sigma^-_0
(n,0)\sigma^+_0 (n,0)} \, (1+r^-_n (z))^{1/2}(1+r^+_n (z))^{1/2},
\end{equation*}
where $\sqrt{\sigma^-_0 (n,0)\sigma^+_0 (n,0)}$ is a square root of
$\sigma^-_0 (n,0)\sigma^+_0 (n,0)$ and  $(1+w)^{1/2} $ is defined by
its Taylor series about $w=0.$ Then $\sqrt{\beta_n^- (z)\beta_n^+
(z)}$ is a well-defined analytic function on $2\mathbb{D}_m, $ so
the basic equation \eqref{be} splits into two equations
\begin{align}\label{E1}
z- \alpha_n(z)- \sqrt{\beta_n^-(z)\beta_n^+(z)} &=0, \\
\label{E2} z- \alpha_n(z)+  \sqrt{\beta_n^-(z)\beta_n^+(z)} &=0.
\end{align}

Next we show that for large enough $m$ equation \eqref{E1} has at
most one root in the disc $\mathbb{D}_m. $  Let
$$
\varphi_n (z) = \alpha_n(z) + \sqrt{\beta_n^-(z)\beta_n^+(z)}, \quad
f_n (z) = z -\varphi_n (z).
$$
By (\ref{7.1}) we have $\alpha^\prime_n (z) = O(1/m^2)$  for
$|z|\leq 1/4.$ On the other hand,  Lemma~\ref{lem4} implies that
$$
\sqrt{\beta_n^-(z)\beta_n^+(z)} = O (1/m^2) \quad \text{for} \quad
z\in 2\mathbb{D}_m,
$$
so by the Cauchy formulas for the derivatives we have
$$
\frac{d}{dz} \sqrt{\beta_n^-(z)\beta_n^+(z)} = O(1/m) \quad
\text{for} \quad z\in \mathbb{D}_m.
$$
Therefore for large enough $m$
$$
\sup \{|\varphi_n^\prime (z)|\, : \; z \in \mathbb{D}_m\} \leq 1/2,
$$
which implies
$$
|\varphi_n (z_1) - \varphi_n (z_2) |   = \left |\int_{z_1}^{z_2}
\varphi_n^\prime (z) dz \right | \leq \frac{1}{2} \, |z_1 -
z_2|\quad \text{for} \;\; z_1, z_2 \in \mathbb{D}_m.
$$
Now we obtain, for $z_1, z_2 \in \mathbb{D}_m,$ that
$$
\begin{aligned}
|f_n (z_1) - f_n (z_2) | &= |(z_1 + \varphi_n (z_1)) - (z_2 +
\varphi_n (z_2))| \\
&\geq |z_1 - z_2| - |\varphi_n (z_1) - \varphi_n (z_2) | \geq
\frac{1}{2} |z_1 - z_2|.
\end{aligned}
$$
Hence the equation $f_n (z) = 0 $ (i.e., equation (\ref{E1})) has at
most one solution in the disc $\mathbb{D}_m.$ Of course, the same
argument gives that equation (\ref{E1}) also has at most one
solution in the disc $\mathbb{D}_m.$

On the other hand, we know by Lemmas~\ref{loc} and \ref{lem2} that
for large enough $m$  equation \eqref{be} has exactly two roots
$z_n^-, z_n^+$ in the disc $|z|\leq 1/2$, so either $z_n^-$ is the
root of (\ref{E1}) and $ z_n^+$ is the root of (\ref{E2}), or vise
versa $z_n^+$ is the root of (\ref{E1}) and $ z_n^-$ is the root of
(\ref{E2}). Therefore, we obtain
\begin{equation*}
z_n^+ - z_n^- -[\alpha_n(z_n^+)- \alpha_n(z_n^-)]  = \pm \left[
\sqrt{\beta_n^-(z_n^+)\beta_n^+(z_n^+)} +
\sqrt{\beta_n^-(z_n^-)\beta_n^+(z_n^-)}\right].
\end{equation*}
Now \eqref{D1}, \eqref{beta+}, \eqref{Sigma0+}, \eqref{beta-} and
\eqref{Sigma0-} imply, for $n= 2m+1,$
\begin{align*}
\gamma_n \left[ 1+ O\left( \frac{1}{m^2} \right) \right]= \pm 2
\frac{\sqrt{(Ab)^m (aB)^{m+1}}}{4^{2m} (m!)^2 } \left[  1 + O \left(
\frac{\log^2 m}{m^2}\right) \right],
\end{align*}
which yields \eqref{gamma+}.

The same argument shows that \eqref{D1}, \eqref{beta+},
\eqref{Sigma0+}, \eqref{beta-} and \eqref{Sigma0-} imply
\eqref{gamma-}.

\end{proof}

\end{document}